\documentclass[12pt,a4paper,reqno]{amsart} 
\usepackage{amssymb}
\usepackage{latexsym}
\usepackage{amsmath}
\usepackage{mathrsfs}
\usepackage{calc}                         

\newcommand{\norm}[1]{\Vert#1\Vert }

\newtheorem{thm}{Theorem}

\newtheorem{cor}[thm]{Corollary}
\newtheorem{lem}[thm]{Lemma}

\theoremstyle{definition}

\theoremstyle{remark}

\newtheorem{rem}[thm]{Remark}

\title[]{Unconditionally convergent multipliers and Bessel sequences}
\author{Carmen Fern\'{a}ndez, Antonio Galbis and Eva Primo}

\begin{document}

\maketitle

\begin{abstract}
We prove that every unconditionally summable sequence in a Hilbert space can be factorized as the product of a square summable scalar sequence and a Bessel sequence. Some consequences on the representation of unconditionally convergent multipliers are obtained, thus providing positive answers to a conjecture by Balazs and Stoeva in some particular cases.
\end{abstract}

\section{Introduction}
A {\it multiplier} on a separable Hilbert space $H$ is a bounded operator
$$
M_{m,\Phi,\Psi}:H\to H,\ f\mapsto\sum_{n=1}^\infty m_n\left<f, \Psi_n\right> \Phi_n,
$$ where $\Phi = \left(\Phi_n\right)_n$ and $\Psi = \left(\Psi_n\right)_n$ are sequences in $H$ and $m = (m_n)_n$ is a scalar sequence called the symbol. The multiplier is said to be unconditionally convergent if the above series converges unconditionally for every $f\in H.$ For any (unconditionally convergent) multiplier $M_{m,\Phi,\Psi}$ its adjoint $M_{\overline{m},\Psi,\Phi}$ is also a (unconditionally convergent) multiplier.
\par\medskip
Observe that each bounded operator $T$ on $H$ can be expressed as a multiplier: if $\left(u_n\right)_n$ is an orthonormal basis, we can take  $\Phi_n=Tu_n,$ $\Psi_n=u_n $ (alternatively  $\Phi_n=u_n,$ $\Psi_n=T^\ast u_n $) and $m_n=1$ for each $n\in {\mathbb N}.$
\par\medskip
In the case that $\Phi = \left(\Phi_n\right)_n$ and $\Psi = \left(\Psi_n\right)_n$ are Bessel sequences in $H$ and $m\in \ell^\infty$ the operator $M_{m,\Phi,\Psi}$ is called a Bessel multiplier. Recall that $\Psi = \left(\Psi_n\right)_n$ is called a {\it Bessel sequence} if there is a constant $B > 0$ such that
$$
\sum_{n=1}^\infty\left|\left<f, \Psi_n\right>\right|^2 \leq B \|f\|^2
$$ for every $f\in H.$ It turns out that $\left(\Psi_n\right)_n$ is a Bessel sequence if and only if there exists a bounded operator $T:\ell^2\to H$ such that $T(e_n) = \Psi_n,$ where $\left(e_n\right)_n$ denote the canonical unit vectors of $\ell^2$ (\cite[Theorem 3.2.3]{chris}).
\par\medskip
Bessel multipliers were introduced and studied in a systematic way by Balazs \cite{balazs} as a generalization of the Gabor multipliers considered in \cite{fn}. In \cite{balazs} it is proved that each Bessel multiplier is unconditionally convergent. Balazs and Stoeva \cite{sb2} provide examples of non-Bessel sequences and non-bounded symbols defining unconditionally convergent multipliers. However all the examples are obtained from a Bessel multiplier after some trick. In fact, Balazs and Stoeva {\it conjecture} in \cite{sb} that every unconditionally convergent multiplier can be written as a Bessel multiplier with constant symbol by shifting weights. More precisely, if $M_{m,\Phi,\Psi}:H\to H$ is an unconditionally convergent multiplier, then they conjecture that there exist scalar sequences $(a_n)
_n,\
(b_n)_n$ such
that
 $$
 m_n = a_n\cdot \overline{b}_n
 $$ and
 $$
 \left(a_n \Phi_n\right)_n,\ \left(b_n \Psi_n\right)_n
 $$ are  Bessel sequences  in $H.$ Several classes of multipliers for which the conjecture is true are obtained in \cite{sb}.
\par\medskip
In the particular case that $m_n = 1$ and $\Psi_n = g$ for every $n\in {\mathbb N},$ the conjecture has a positive answer if and only if for every unconditionally summable sequence $(\Phi_n)_n$ in a separable Hilbert space $H$ we may find $(\alpha_n)_n\in \ell^2$ such that $(\frac{1}{\alpha_n }\Phi_n)_n$ is a Bessel sequence in $H.$ So, the main aim of the present paper is to analyze the structure of unconditionally summable sequences in a separable Hilbert space. As a consequence we obtain some new situations where the conjecture of Balasz and Stoeva is still true, which are different in spirit to the ones considered in \cite{sb}.

\section{Results}

We will use the well known fact that a series $\sum_{n=1}^\infty x_n$ in a Banach space $X$ is unconditionally convergent if and only if there exist a compact operator $T:c_0\to X$ with the property that $T(e_n) = x_n,$ where $\left(e_n\right)_n$ denote the canonical unit vectors of $c_0$ (see for instance the omnibus theorem on unconditional summability in \cite[1.9]{djt}). We recall that, in the case that $X = H$ is a Hilbert space, every bounded operator $T:c_0\to H$ is compact. In fact, the closed unit ball $B$ of $H$ is weakly compact, the transposed map $T^\ast:H\to \ell^1$ is a bounded operator and weak and norm convergence of sequences in $\ell^1$ coincide (\cite[Theorem 1.7]{djt}). Therefore $T^\ast$ is a compact operator and so is $T.$
\par\medskip
From the previous considerations we conclude that a series $\sum_{n=1}^\infty x_n$ in a Hilbert space $H$ is unconditionally convergent if and only if there exists a bounded operator $T:c_0\to H$ with the property that $T(e_n) = x_n.$ An important consequence is the fact that the unconditionally convergence of $\sum_{n=1}^\infty x_n$ is equivalent to
$$
\sum_{n=1}^\infty\left|\left<x_n,g\right>\right| < \infty\ \ \forall g\in H.
$$ This is so because if the last condition is satisfied then, by closed graph theorem, $S:H\to \ell^1,$ $S(g):=\left(\left<x_n,g\right>\right)_n,$ defines a bounded operator and $T = S^\ast:\ell^\infty\to H$ verifies $T(e_n) = x_n.$
\par\medskip
For a fixed sequence $\alpha = (\alpha_n)_n$ the diagonal operator $D_\alpha$ acts on a sequence $x = (x_n)_n$ as
$$
 D_\alpha (x) = \left(\alpha_n x_n\right)_n.
 $$ If $\alpha\in\ell^2$ then $D_\alpha:\ell^\infty \to \ell^2$ is a bounded operator, while $D_\alpha:\ell^2 \to \ell^2$ whenever $\alpha\in c_0.$

 \par\medskip
  \begin{lem}\label{lem:reformulacion} The following statements are  equivalent:
 \begin{itemize}
  \item[(a)] Every unconditionally summable sequence  $\left(\Phi_n\right)_n$ in $H$ can be written as  $\Phi_n = \alpha_n f_n,$ where $(\alpha_n)_n\in \ell^2$ and  $\left(f_n\right)_n$ is a Bessel sequence in $H.$
  \item[(b)] Every bounded operator  $T:c_0\to H$ can be  factorized as  $$T = A \circ D_\alpha$$ where  $D_\alpha:c_0\to \ell^2$ is a diagonal operator and  $A:\ell^2\to H$ is a bounded operator.
 \end{itemize}
 \end{lem}
\begin{proof}
 $(a)\Rightarrow (b).$ If $T:c_0\to H$ is bounded then  $\left(\Phi_n\right)_n = \left(T(e_n)\right)_n$  is unconditionally summable (\cite[Theorem 1.9]{djt}), hence $T(e_n) = \alpha_n f_n,$ where $\alpha = (\alpha_n)_n\in \ell^2$ and $\left(f_n\right)_n$ is a Bessel sequence in $H.$ Therefore $\left(f_n\right)_n$ defines a bounded operator  $A:\ell^2 \to H,\ A(\beta) = \sum_n \beta_n f_n,$ and  $T = A\circ D_\alpha.$
 \par
 $(b)\Rightarrow (a).$ Let $\left(\Phi_n\right)_n$ be an unconditionally summable sequence in $H.$ Then there is a bounded operator $T:c_0\to H$ such that $T(e_n) = \Phi_n$ and, by hypothesis, it can be factorized as $T = A \circ D_\alpha$ where $\alpha \in \ell^2$ and $A:\ell^2\to H$ is a bounded operator. Then $\left(f_n\right)_n:= \left(A(e_n)\right)_n$ is a  Bessel sequence in  $H$ and clearly $\Phi_n = A\left(\alpha_n e_n\right) = \alpha_n f_n.$ $\Box$
 \end{proof}

\begin{rem} In the same spirit, the conjecture in \cite{sb} has a positive answer for a given unconditionally convergent  multiplier  $M_{m,\Phi,\Psi}$   on $\ell^2$ if and only if
the continuous  bilinear map
$$
T: c_0 \times \ell^2\to\ell^2,\ \ (\alpha, f)\mapsto \sum_{n=1}^\infty \alpha_n m_n \left< f, \Psi_n\right>\Phi_n,
$$ admits a factorization $$T = B \circ D$$ where $B:\ell^2\to \ell^2$ is a bounded operator and
$D:c_0 \times \ell^2\to \ell^2$ is a continuous bilinear map such that  for every $f \in \ell^2,$ $D(\cdot, f):c_0 \to \ell^2$ is a diagonal operator.

\end{rem}

 We recall that any bounded operator $B:c_0\to \ell^\infty,\ B(e_j) = \left(b^i_j\right)_i,$ has the property that $b^i:= \left(b^i_j\right)_j \in \ell^1$ for every $i\in {\mathbb N}$ and $$
 \|B\| = \sup_i \|b^i\|_{\ell^1}.
 $$

The next result can be viewed as an improvement of Orlicz's Theorem (see for instance \cite[Theorem 1.11]{djt} or \cite[Theorem 3.16]{h}) that every unconditionally summable sequence in a Hilbert space is absolutely 2-summable. It is the main result of the paper.

 \begin{thm}\label{th:sequence} Every unconditionally summable sequence  $\left(\Phi_n\right)_n$ in a separable Hilbert space  $H$ can be expressed as  $\Phi_n = \overline{a}_n f_n,$ where $(a_n)_n\in \ell^2$ and $\left(f_n\right)_n$ is a  Bessel sequence in $H.$
 \end{thm}
    \begin{proof}
    By Lemma \ref{lem:reformulacion} it is enough to show that every bounded operator $T:c_0\to H$ can be factorized as $T=A\circ D_\alpha,$ where $\alpha \in \ell^2$ and $A:\ell^2\to H$ is a bounded operator. According to  \cite[3.7 and 5.9]{djt}, $T$ is a 2-integral operator, hence it is 2-nuclear (\cite[Theorem 5]{p}). Therefore there are bounded operators $B:c_0\to \ell^\infty,$ $S:\ell^2\to H$ and $\lambda \in \ell^2$ such that $T = S\circ D_\lambda\circ B$ \cite[Theorem 19.7.4]{jarchow}. To finish it suffices to find $\alpha \in \ell^2$ and a bounded operator $\tilde{A}$ on $\ell^2$ such that
 $D_\lambda\circ B=\tilde{A}\circ D_\alpha,$ since then
    $$
    T =A\circ D_\alpha,
    $$ with $A=S\circ \tilde{A}.$ \par\medskip
As  $D_\lambda\circ B = D_{t\lambda}\circ (t^{-1}B)$ for each $t > 0,$ without loss of generality we can assume $\|B\| = 1.$ We denote $B(e_j) = \left(b^i_j\right)_i$ and $b^i:= \left(b^i_j\right)_j \in \ell^1.$ We define  $\alpha = (\alpha_k)_k$ such that
    $$
    |\alpha_k|^2 := \sum_{i=1}^\infty |\lambda_i|^2\cdot |b^i_k|.
    $$ Then
\begin{equation}\label{eq:alphak}
\sum_{k=1}^\infty |\alpha_k|^2 = \sum_{i=1}^\infty |\lambda_i|^2\cdot \|b^i\|_{\ell^1} \leq \|\lambda\|_{\ell^2}^2,
\end{equation} hence $\alpha \in \ell^2.$ Next, we consider
    $$
    f_k:= \frac{1}{\alpha_k}\left(\lambda_i b^i_k\right)_i,\ \ k\in {\mathbb N}.
    $$ Since $|b^i_k|^2\leq |b^i_k|,$ the inequality (\ref{eq:alphak}) implies that $f_k\in \ell^2. $ To finish the proof, we have to show that there is a bounded operator $\tilde{A}$ on $\ell^2$ such that $\tilde{A}(e_k) = f_k,$ that is, $\left(f_k\right)_k$ is a Bessel sequence in $\ell^2.$ To this end, we fix $\beta = (\beta_k)_k\in \ell^2$ and  $\gamma = (\gamma_k)_k\in \ell^2.$ Then,
   $$
    \sum_{k=1}^N \left|\beta_k\left<f_k, \gamma\right>\right| \leq  \sum_{k=1}^N \frac{|\beta_k|}{|\alpha_k|}\sum_{j= 1}^\infty|\lambda_j \gamma_j|\cdot |b^j_k|
    $$ for all $N\in {\mathbb N}.$  As $\lambda, \, \gamma \in \ell^2,$
    $$
    \begin{array}{ll}
     \begin{displaystyle}
      \sum_{j= 1}^\infty|\lambda_j \gamma_j|\cdot |b^j_k|
     \end{displaystyle} & \begin{displaystyle} \leq \left(\sum_{j= 1}^\infty|\lambda_j|^2\cdot |b^j_k|\right)^{\frac{1}{2}}\cdot \left(\sum_{j= 1}^\infty|\gamma_j|^2\cdot |b^j_k|\right)^{\frac{1}{2}}\end{displaystyle}\\ & \\ & = \begin{displaystyle}
     |\alpha_k|\cdot \left(\sum_{j= 1}^\infty |\gamma_j|^2\cdot|b^j_k|\right)^{\frac{1}{2}}.\end{displaystyle}
         \end{array}
    $$ Moreover
    $$
          \sum_{k=1}^\infty\sum_{j=1}^\infty |\gamma_j|^2\cdot|b^j_k|
          = \sum_{j=1}^\infty |\gamma_j|^2\cdot \|b^j\|_{\ell^1} \leq \|\gamma\|_{\ell^2}^2.
    $$ This means that
    $$
    \left(\left(\sum_{j= 1}^\infty|\gamma_j|^2\cdot |b^j_k|\right)^{\frac{1}{2}}\right)_{k\in {\mathbb N}}\in \ell^2.
    $$ Hence,
    $$
    \sum_{k=1}^\infty\left|\beta_k\left<f_k, \gamma\right>\right|
    $$ is less than or equal to
    $$
    \sum_{k=1}^\infty |\beta_k|\cdot \left(\sum_{j= 1}^\infty|\gamma_j|^2\cdot |b^j_k|\right)^{\frac{1}{2}} < \infty.
    $$ Since this holds for every  $\beta\in \ell^2$ we conclude that
    $$
    \sum_{k=1}^\infty\left|\left<f_k, \gamma\right>\right|^2 < \infty
    $$ for every $\gamma\in \ell^2.$ Now, the closed graph theorem gives the conclusion. $\Box$

    \end{proof}

Theorem \ref{th:sequence} gives a positive answer to the conjecture when $(\Psi_n)_n$ is a constant sequence.  Next we consider a  more general situation.

\begin{cor}\label{cor:without-accum} Let $M_{m,\Phi,\Psi}$ be an unconditionally convergent multiplier and assume that $0$ is not a weak accumulation point of the sequence $\left(\frac{\Psi_n}{\|\Psi_n\|}\right)_n.$ Then there exist scalar sequences $(a_n)_n,\ (b_n)_n$ such
that $m_n = a_n\cdot \overline{b}_n
 $ and
 $
 \left(a_n \Phi_n\right)_n,\ \left(b_n \Psi_n\right)_n
 $ are  Bessel sequences  in $H.$
 \end{cor}
\begin{proof}
 In fact, our hypothesis implies the existence of finitely many elements $f_1, \ldots, f_K\in H$ with the property that
$$
\sum_{k=1}^K\left|\left<f_k,\frac{\Psi_n}{\|\Psi_n\|}\right>\right| \geq 1
$$ for every $n\in {\mathbb N}.$ Since $M_{m,\Phi,\Psi}$ is an unconditionally convergent multiplier we have
$$
\sum_{n=1}^\infty\left|\left<f, \Psi_n\right>\right|\cdot \left|\left<\Phi_n, g\right>\right| < \infty
$$ for every $f,g\in H.$ Consequently
$$
\sum_{n=1}^\infty\left|m_n\right|\cdot\norm{\Psi_n}\cdot\left|\left<\Phi_n, g\right>\right| \leq \sum_{k=1}^K\sum_{n=1}^\infty\left|\left<f_k, \Psi_n\right>\right|\cdot \left|\left<\Phi_n, g\right>\right| < \infty
$$ for every $g\in H.$ It follows that the series $\sum_{n=1}^\infty m_n \|\Psi_n\|\Phi_n$ is unconditionally convergent and we can apply Theorem \ref{th:sequence} to find a sequence $(b_n)_n\in \ell^2$ such that $\left(\frac{m_n}{b_n}\|\Psi_n\|\Phi_n\right)_n$ is a Bessel sequence. Since also $\left(b_n \frac{\Psi_n}{\|\Psi_n\|}\right)_n$ is a Bessel sequence, the conclusion follows. 
\end{proof}
By Orlicz's Theorem, in the case that $(\Psi_n)_n$ is constant, the unconditional convergence of the series $\begin{displaystyle}\sum_{n=1}^\infty m_n\left<f, \Psi_n\right> \Phi_n\end{displaystyle}$ implies that $$
  \sum_{n=1}^\infty (|m_n|\|\Phi_n\|\|\Psi_n\|)^2 < \infty.$$ In particular, the sequence $\left(m_n\cdot \|\Phi_n\|\cdot \|\Psi_n\|\right)_n$ converges to zero and \cite[Proposition 1.1]{sb} cannot be applied.
Obviously, in  Corollary \ref{cor:without-accum}, the condition on the sequence $(\Psi_n)_n$ can be replaced by a similar condition on $(\Phi_n)_n.$  The following result shows that the conjecture stated by Balazs and Stoeva in \cite{sb} holds under the stronger hypothesis of absolute convergence of the series. The proof depends on Theorem \ref{th:sequence} and it does not follow from the results in \cite{sb}.

\begin{thm}\label{th:absolute} Let $M_{m,\Phi,\Psi}$ be  such that for each  $f\in H,$ the series
$$
\sum_{n=1}^\infty m_n\left<f, \Psi_n\right> \Phi_n$$ converges absolutely in $H.$ Then there exist scalar sequences $(a_n)_n$ and $(b_n)_n$ such that $m_n = a_n\cdot \overline{b_n},$ and $\left( b_n \Psi_n \right)_n$  and $\left( a_n \Phi_n \right)_n$ are   Bessel sequences in $H. $

\end{thm}
\begin{proof}
Without loss of generality, we may assume that $ ||\Phi_n||=1$ for every $n\in {\mathbb N}.$
 The condition $(m_n\left< f, \Psi_n \right>)_n \in \ell^1$ for every $f\in H$ implies that the sequence $(\overline{m_n} \Psi_n)_n$ is unconditionally summable in $H$, therefore by Theorem \ref{th:sequence}, there is $(c_n)_n \in \ell^2$ such that  $\left( \frac{\overline{m_n}}{c_n} \Psi_n \right)_n$ is a Bessel sequence. As $\left( \overline{c_n}\Phi_n \right)_n$ is also a Bessel sequence, we conclude. $\Box$
\end{proof}

As a consequence of the previous result, it is easy to see that the  absolute convergence of  $\sum_{n=1}^\infty m_n\left<f, \Psi_n\right> \Phi_n$ (for every $f$ in $H$)  implies  the unconditional convergence of the series $$
\sum_{n=1}^\infty m_n \Psi_n\otimes \Phi_n
$$ in the Hilbert space $S^2(H)$ of Hilbert-Schmidt operators on $H.$

Let $B_H$ denote the closed unit ball of $H$ endowed with the weak topology and $\mu$ a probability Borel measure on $B_H.$ Then we have the canonical continuous inclusion
$$
j_\mu:H\to L^2(B_H,\mu),\ j_\mu(f)\left(g\right):=\left<f,g\right>.
$$

\begin{thm} Let $B_H$ denote the closed unit ball of $H$ endowed with the weak topology and assume that the multiplier $M_{m,\Phi,\Psi}$ has the additional property that the series
$$
\sum_{n=1}^\infty m_n \Psi_n\otimes \Phi_n
$$ converges unconditionally in  $S^2(H).$  Then, for every probability Borel measure $\mu$ on $B_H$ there exist scalar sequences $(a_n)_n,\ (b_n)_n$ such that $m_n = a_n\cdot \overline{b}_n,$ $\left(a_n \Psi_n\right)_n$ is a Bessel sequence in $H$ and $\left(j_\mu(b_n \Phi_n)\right)_n$ is a Bessel sequence in $L^2(B_H,\mu).$ In particular
$$
\sum_{n=1}^\infty\left|\left<f, b_n \Phi_n\right>\right|^2 < \infty
$$ for $\mu$-almost every $f\in B_H.$
\end{thm}
\begin{proof}
In fact, according to Theorem \ref{th:sequence} there is $(\alpha_n)_n\in\ell^2$ such that
$$
\left(\frac{m_n}{\alpha_n}\Psi_n\otimes \Phi_n\right)_n
$$ is a Bessel sequence in $S^2(H).$ In particular, for some constant $C > 0,$
\begin{equation}\label{eq:remark}
\sum_{n=1}^\infty \left|\frac{m_n}{\alpha_n} \left<f, \Psi_n\right> \left<g, \Phi_n\right>\right|^2 \leq C \|f\|^2\cdot \|g\|^2
\end{equation} for every $f,g\in B_H.$ We now consider
$$
a_n^2:= \left|\frac{m_n}{\alpha_n}\right|^2 \int_{B_H}\left|\left<g, \Phi_n\right>\right|^2\ d\mu(g).
$$ After integrating in (\ref{eq:remark}) we obtain that $\left(a_n \Psi_n\right)_n$ is a Bessel sequence. Moreover, for $b_n = \frac{\overline{m}_n}{a_n}$ we have
$$
\sum_{n=1}^\infty\int_{B_H}\left|\left<f, b_n \Phi_n\right>\right|^2\ d\mu(f) = \sum_{n=1}^\infty \alpha_n^2 < \infty,
$$ from where the conclusion follows. $\Box$
\end{proof}

\end{document}